\documentclass[11pt]{article}
\usepackage[cp1251]{inputenc}
\usepackage[english]{babel}
\usepackage{amsmath}
\usepackage{amsfonts}
\usepackage{mathrsfs}
\usepackage{amssymb}
\usepackage{graphicx}
\usepackage{multirow}
\usepackage{hanging}
\usepackage{nccmath}
\usepackage[flushmargin, symbol*]{footmisc}
\usepackage[usenames,dvipsnames]{color}
\newtheorem{theorem}{\textsf{Theorem}}
\newtheorem{lemma}{\textsf{Lemma}}
\newtheorem{proposition}{\textsf{Proposition}}
\newenvironment{proof}[1][\textsf{Proof. }]{\textbf{#1}}{$\square$}

\newenvironment{remark}[1][\textsf{Remark. }]{\textbf{#1}}{}
\begin{document}

\title{
\date{}
{
\large \textsf{\textbf{The dual Jacobian of a generalised hyperbolic tetrahedron,\\ and volumes of prisms}}
}
}
\author{\small Alexander Kolpakov \hspace*{42pt}
\small Jun Murakami}
\maketitle

\begin{abstract}\noindent

We derive an analytic formula for the dual Jacobian matrix of a generalised hyperbolic tetrahedron. Two cases are considered: a mildly truncated and a prism truncated tetrahedron. The Jacobian for the latter arises as an analytic continuation of the former, that falls in line with a similar behaviour of the corresponding volume formulae.

Also, we obtain a volume formula for a hyperbolic $n$-gonal prism: the proof requires the above mentioned Jacobian, employed in the analysis of the edge lengths behaviour of such a prism, needed later for the Schl\"{a}fli formula.

\end{abstract}

\parindent=0pt 

\section{Introduction}\label{introduction}

Let $T$ be a generalised hyperbolic tetrahedron (in the sense of \cite{MU, U}) depicted in Fig.~\ref{fig_tetrahedron}. If the truncating planes associated with its ultra-ideal vertices do not intersect, we call such a tetrahedron \textit{mildly truncated}, otherwise we call it \textit{intensely truncated}. If only two of them intersect, we call such a tetrahedron \textit{prism truncated} \cite{KM2012}. Let us note that a prism truncated orthoscheme is, in fact, a Lambert cube \cite{K}.

The volumes of the tetrahedron and its truncations are of particular interest, since they are the simplest representatives of hyperbolic polyhedra. Over the last decade an extensive study produced a number of volume formulae suitable for analytic and numerical exploration \cite{ChoKim, DM, K, KM2012, MY, U}. 
A similar study was done for the spherical tetrahedron \cite{KMP, M2012}, which can be viewed as a natural counterpart of the hyperbolic one. Many analytic properties of the volume formula for a hyperbolic tetrahedron came into view concerning the Volume Conjecture \cite{Kashaev, HMJM}.

\begin{figure}[ht]
\begin{center}
\includegraphics* [totalheight=5cm]{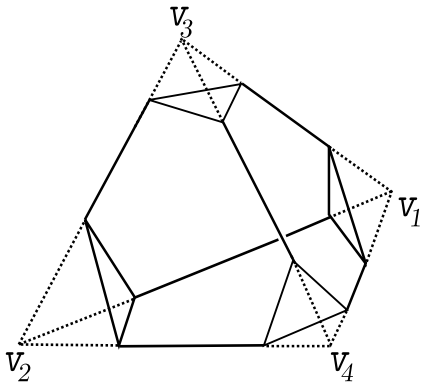}
\end{center}
\caption{Generalised hyperbolic tetrahedron} \label{fig_tetrahedron}
\end{figure}

However, other geometric characteristics of a generalised hyperbolic tetrahedron $T$ are also important and bring some useful information. In particular, $\mathrm{Jac}(T)$, the Jacobian of $T$, which is the Jacobian matrix of the edge length with respect to the dihedral angles, is such. This matrix enjoys many symmetries \cite{Luo} and can be computed out of the Gram matrix of $T$ \cite{Guo}.

In the present paper, we consider $\mathrm{Jac}^{\star}(T)$, \textit{the dual Jacobian} of a generalised hyperbolic tetrahedron $T$. By the dual Jacobian of $T$ we mean the Jacobian matrix of the dihedral angles with respect to the edge length. Such an object behaves nicely when $T$ undergoes both mild and intense truncation: the dual Jacobian of a prism truncated tetrahedron is an analytic continuation for that of a mildly doubly truncated one. Let us mention, that the respective volume formulae are also connected by an analytic continuation, in an analogous manner \cite{KM2012, MU}.

As an application of our technique, we give a volume formula for a hyperbolic $n$-gonal prism, c.f. \cite{DK}.

\medskip

\textbf{Acknowledgements.} The authors gratefully acknowledge financial support provided by the Swiss National Science Foundation (SNSF project no.~P300P2-151316) and the Japan Society for the Promotion of Science (Grant-in-Aid projects no.~25287014, no.~2561002 and Invitation Programs for Research project no.~S-14021). The authors thank the anonymous referee for his/her careful reading of the manuscript and helpful comments. 

\section{Preliminaries}\label{preliminaries}

Let $T$ be a mildly truncated hyperbolic tetrahedron with vertices $\mathrm{v}_k$, $k\in\{1,2,3,4\}$, edges $e_{ij}$ (connecting the vertices $\mathrm{v}_i$ and $\mathrm{v}_j$) with dihedral angles $a_{ij}$ and lengths $\ell_{ij}$, $i,j\in \{1,2,3,4\}$, $i<j$. 

Depending on whether the vertex $\mathrm{v}_k$ is proper ($\mathrm{v}_k \in \mathbb{H}^3$), ideal ($\mathrm{v}_k \in \partial\mathbb{H}^3$) or ultra-ideal ($\mathrm{v}_k$ defines a polar hyperplane as described in \cite[Section 3]{Ratcliffe}, c.f. Theorem 3.2.12), let us set the quantity $\varepsilon_k$ to be $+1$, $0$ or $-1$, respectively. 

For each vertex $\mathrm{v}_i$ of $T$ let us consider the face $F_{jkl}$ opposite to it, where $\{i,j,k,l\} = \{1,2,3,4\}$. The link $L(\mathrm{v}_l)$ of the vertex $\mathrm{v}_l$ is either a spherical triangle ($\varepsilon_l = +1$), a Euclidean triangle ($\varepsilon_l = 0$) or a hyperbolic triangle ($\varepsilon_l = -1$). Let us define the quantity $b^{i}_{jk}$ as follows:
\begin{equation*}
b^{i}_{jk} := \left\{\begin{array}{l}
\mbox{the plane angle of $F_{jkl}$ opposite to the edge $e_{jk}$, if $\varepsilon_l = +1$};\\
\mbox{zero, if $\varepsilon_l = 0$};\\
\mbox{the length of the common perpendicular to the edges $e_{jl}$}\\
\mbox{and $e_{kl}$ of $F_{jkl}$, if $\varepsilon_l = -1$}.
\end{array}\right.
\end{equation*}

Here, we consider the face $F_{jkl}$ as a generalised hyperbolic triangle, for which the trigonometric laws hold as described in \cite{Cho, LuoGuo}.

Let us also define a quantity $\mu^{i}_{jk}$ by means of the formula 
\begin{equation*}
\mu^{i}_{jk} := \int^{b^{i}_{jk}}_{0} \cos(\sqrt{\varepsilon_l}s) \mathrm{d}s.
\end{equation*}
Let ${\mu^{\prime}}^{i}_{jk}$ denote the derivative of $\mu^{i}_{jk}$ with respect to $b^{i}_{jk}$, which means that
\begin{equation*}
{\mu^{\prime}}^{i}_{jk} = \cos(\sqrt{\varepsilon_l}b^{i}_{jk}).
\end{equation*}
Let $\sigma_{kl}$ denote the following quantity associated with an edge $e_{kl}$, $k,l\in\{1,2,3,4\}$, $k<l$,
\begin{equation*}
\sigma_{kl} := \frac{1}{2} e^{\ell_{kl}} - \frac{1}{2} \varepsilon_k \varepsilon_l e^{-\ell_{kl}}.
\end{equation*}
Let $\sigma^{\prime}_{kl}$ denote the derivative of $\sigma_{kl}$ with respect to $\ell_{kl}$, so we have that
\begin{equation*}
\sigma^{\prime}_{kl} = \frac{1}{2} e^{\ell_{kl}} + \frac{1}{2} \varepsilon_k \varepsilon_l e^{-\ell_{kl}}.
\end{equation*}

Let us define the momentum $M_i$ of the vertex $\mathrm{v}_i$ opposite to the face $F_{jkl}$, $\{i,j,k,l\} = \{1,2,3,4\}$ by the following equality (c.f.~\cite[VII.6]{F}):
\begin{equation*}
M^{i} := \mu^{i}_{jk}\,\mu^{i}_{jl}\,\sigma_{kl}.
\end{equation*}

The quantity above is well defined grace to the following theorem.
\begin{theorem}[The Sine Law for faces]\label{thm_sine_face}
Let $F_{jkl}$ be the face of $T$ opposite to the vertex $\mathrm{v}_i$, $\{i,j,k,l\} = \{1,2,3,4\}$. Then $F_{jkl}$ is a generalised hyperbolic triangle and the following equalities hold:
\begin{equation*}
\frac{\mu^{i}_{jk}}{\sigma_{jk}} = \frac{\mu^{i}_{jl}}{\sigma_{jl}} = \frac{\mu^{i}_{kl}}{\sigma_{kl}}.
\end{equation*}
\end{theorem}

Let us also define the momentum $M_{jkl}$ of the face $F_{jkl}$ opposite to the vertex $\mathrm{v}_i$, $\{i,j,k,l\}=\{1,2,3,4\}$ by setting (c.f.~\cite[VII.6]{F})
\begin{equation*}
M_{jkl} := \mu^{j}_{kl} \sin a_{ik} \sin a_{il}.
\end{equation*}

The quantity above is well defined, according to the following theorem.
\begin{theorem}[The Sine Law for links]\label{thm_sine_link}
Let $\mathrm{v}_i$ be the vertex of $T$ opposite to the face $F_{jkl}$, $\{i,j,k,l\} = \{1,2,3,4\}$. Then $L(\mathrm{v}_i)$ is either a spherical, a Euclidean or a hyperbolic triangle and the following equalities hold:
\begin{equation*}
\frac{\sin a_{ij}}{\mu^{j}_{kl}} = \frac{\sin a_{ik}}{\mu^{k}_{jl}} = \frac{\sin a_{il}}{\mu^{l}_{jk}}.
\end{equation*}
\end{theorem}

Both Theorem~\ref{thm_sine_face} and Theorem~\ref{thm_sine_link} are paraphrases of the spherical, Euclidean or hyperbolic sine laws (for a generalised hyperbolic triangle, see \cite{LuoGuo}). The following theorems are the cosine laws for a generalised hyperbolic triangle adopted to the notation of the present paper.

\begin{theorem}[The first Cosine Law for faces]
Let $F_{jkl}$ be the face of $T$ opposite to the vertex $\mathrm{v}_i$, $\{i,j,k,l\} = \{1,2,3,4\}$. Then $F_{jkl}$ is a generalised hyperbolic triangle and the following equality holds:
\begin{equation*}
\sigma^{\prime}_{kl} = \frac{{\mu^{\prime}}^{i}_{kl} + {\mu^{\prime}}^{i}_{jk}\, {\mu^{\prime}}^{i}_{jl}}{\mu^{i}_{jk}\, \mu^{i}_{jl}}.
\end{equation*}
\end{theorem}

\begin{theorem}[The second Cosine Law for faces]
Let $F_{jkl}$ be the face of $T$ opposite to the vertex $\mathrm{v}_i$, $\{i,j,k,l\} = \{1,2,3,4\}$. Then $F_{jkl}$ is a generalised hyperbolic triangle and the following equality holds:
\begin{equation*}
{\mu^{\prime}}^{i}_{jk} = \frac{-\varepsilon_l \sigma^{\prime}_{jk} + \sigma^{\prime}_{jl} \sigma^{\prime}_{kl}}{\sigma_{jl} \sigma_{kl}}.
\end{equation*}
\end{theorem}

\begin{theorem}[The Cosine Law for links]
Let $\mathrm{v}_i$ be the vertex of $T$ opposite to the face $F_{jkl}$, $\{i,j,k,l\} = \{1,2,3,4\}$. Then $L(\mathrm{v}_i)$ is either a spherical, a Euclidean or a generalised hyperbolic triangle and the following equality holds:
\begin{equation*}
{\mu^{\prime}}^{j}_{kl} = \frac{\cos a_{ij} + \cos a_{ik} \cos a_{il}}{\sin a_{ik} \sin a_{il}}.
\end{equation*}
\end{theorem}

\section{Auxiliary lemmata}\label{lemmata}

In the present section we shall consider various partial derivatives of certain geometric quantities associated with either the faces or the vertex links of a generalised hyperbolic tetrahedron $T$. These derivatives will be used later on in the computation of the entries of $\mathrm{Jac}^{\star}(T)$.

\begin{lemma}\label{lemma_l_b}
For $\{i,j,k,l\} = \{1,2,3,4\}$ we have
\begin{equation*}
\frac{\partial \ell_{kl}}{\partial b^{i}_{kl}} = - \varepsilon_j\,\frac{\mu^{i}_{kl}}{M^{i}},
\end{equation*}
\begin{equation*}
\frac{\partial \ell_{kl}}{\partial b^{i}_{jk}} = - \sigma^{\prime}_{jl}\,\frac{\mu^{i}_{kl}}{M^{i}},\hspace*{0.10in} \frac{\partial \ell_{kl}}{\partial b^{i}_{jl}} = - \sigma^{\prime}_{jk}\,\frac{\mu^{i}_{kl}}{M^{i}}.
\end{equation*}
\end{lemma}
\begin{proof}
According to the definition of $\sigma_{kl}$, we have $\sigma_{kl} = 0$ only in the following two cases: $\varepsilon_k = \varepsilon_l = +1$ and $\ell_{kl} = 0$, or $\varepsilon_k = \varepsilon_l = -1$ and $\ell_{kl}=0$. In the former case, we have a degenerate tetrahedron with two proper vertices collapsing to one point. In the latter case the tetrahedron has two ultra-ideal vertices, whose polar planes are tangent at a point on the ideal boundary $\partial \mathbb{H}^3$. This is a limiting case, since in a generalised (mildly truncated) tetrahedron two polar planes never intersect or become tangent. Thus, we suppose that $\sigma_{kl} \neq 0$.

By taking derivatives on both sides of the first Cosine Law for faces, we get the following formulae: 
\begin{equation*}
\sigma_{kl}\,\frac{\partial \ell_{kl}}{\partial b^{i}_{kl}} = \frac{\partial \sigma^{\prime}_{kl}}{\partial b^{i}_{kl}} = \frac{1}{\mu^{i}_{jk} \mu^{i}_{jl}} \frac{\partial {\mu^{\prime}}^{i}_{kl}}{\partial b^{i}_{kl}}= -\varepsilon_j\, \frac{\mu^{i}_{kl}}{\mu^{i}_{jk} \mu^{i}_{jl}},
\end{equation*}
since 
\begin{equation*}
\frac{\partial \sigma^{\prime}_{kl}}{\partial b^{i}_{kl}} = \sigma_{kl}\,\frac{\partial \ell_{kl}}{\partial b^{i}_{kl}}\hspace*{0.05in} \mbox{ and } \hspace*{0.05in}\frac{\partial {\mu^{\prime}}^{i}_{kl}}{\partial b^{i}_{kl}}= -\varepsilon_j\, \mu^{i}_{kl}
\end{equation*}
by a direct computation. This implies the first identity of the lemma.

Now we compute 
\begin{equation*}
\sigma_{kl}\,\frac{\partial \ell_{kl}}{\partial b^{i}_{jk}} = \frac{\partial \sigma^{\prime}_{kl}}{\partial b^{i}_{jk}} = - \frac{(({\mu^{\prime}}^{i}_{jk})^{2} + \varepsilon_l (\mu^{i}_{jk})^{2})\mu^{i}_{jl}{\mu^{\prime}}^{i}_{jl} + \mu^{i}_{jl}{\mu^{\prime}}^{i}_{jk}{\mu^{\prime}}^{i}_{kl}}{(\mu^{i}_{jk} \mu^{i}_{jl})^{2}} = 
\end{equation*}
\begin{equation*}
- \frac{{\mu^{\prime}}^{i}_{jl} + {\mu^{\prime}}^{i}_{jk}{\mu^{\prime}}^{i}_{kl}}{\mu^{i}_{jk} \mu^{i}_{kl}}\cdot \frac{\mu^{i}_{kl}}{\mu^{i}_{jk} \mu^{i}_{jl}} = - \sigma^{\prime}_{jl}\,\frac{\mu^{i}_{kl}}{\mu^{i}_{jk} \mu^{i}_{jl}}.
\end{equation*}
where we use the identity $({\mu^{\prime}}^{i}_{jk})^{2} + \varepsilon_l (\mu^{i}_{jk})^{2} = 1$ and, as before, the fact that $\frac{\partial {\mu^{\prime}}^{i}_{jk}}{\partial b^{i}_{jk}} = -\varepsilon_l \mu^{i}_{jk}$. Then the second identity follows. The third one is analogous to the second one under the permutation of the indices $k$ and $l$.
\end{proof}

\begin{lemma}\label{lemma_b_a}
For $\{i,j,k,l\} = \{1,2,3,4\}$ we have
\begin{equation*}
\frac{\partial b^{j}_{kl}}{\partial a_{ij}} = \varepsilon_i\, \frac{\sin a_{ij}}{M_{jkl}}
\end{equation*}
\begin{equation*}
\frac{\partial b^{j}_{kl}}{\partial a_{ik}} = \varepsilon_i\, \frac{\sin a_{ij}}{M_{jkl}}\,{\mu^{\prime}}^{l}_{jk},\hspace*{0.1in} \frac{\partial b^{j}_{kl}}{\partial a_{il}} = \varepsilon_i\, \frac{\sin a_{ij}}{M_{jkl}}\,{\mu^{\prime}}^{k}_{jl}.
\end{equation*}
\end{lemma}
\begin{proof}
By taking derivatives on both sides of the Cosine Law for links, we get the following formulae:
\begin{equation*}
-\varepsilon_i\, \mu^{j}_{kl}\, \frac{\partial b^{j}_{kl}}{\partial a_{ij}} = \frac{\partial {\mu^{\prime}}^{j}_{kl}}{\partial a_{ij}} = - \frac{\sin a_{ij}}{\sin a_{ik} \sin a_{il}}.
\end{equation*}
The first identity of the lemma follows.

Then we subsequently compute
\begin{equation*}
-\varepsilon_i\, \mu^{j}_{kl}\, \frac{\partial b^{j}_{kl}}{\partial a_{ik}} = \frac{\partial {\mu^{\prime}}^{j}_{kl}}{\partial a_{ik}} = - \frac{\cos a_{il} + \cos a_{ij} \cos a_{ik}}{\sin a_{ij} \sin a_{ik}}\, \frac{\sin a_{ij}}{\sin a_{ik} \sin a_{il}} = 
\end{equation*}
\begin{equation*}
{\mu^{\prime}}^{l}_{jk}\, \frac{\sin a_{ij}}{\sin a_{ik} \sin a_{il}}.
\end{equation*}
The second identity follows. The third one is analogous under the permutation of the indices $k$ and $l$.
\end{proof}

Now we shall prove several identities that relate the principal minors $G_{ii}$, $i\in\{1,2,3,4\}$ of the Gram matrix $G := G(T)$ of the tetrahedron $T$ with its face or vertex momenta.  

\begin{lemma}\label{lemma_detGii}
For $\{i,j,k,l\} = \{1,2,3,4\}$, we have that
\begin{equation*}
\det G_{ii} = \varepsilon_i\,M^2_{jkl}.
\end{equation*}
\end{lemma}
\begin{proof}
Let us perform the computation for $G_{11}$ and other cases will follow by analogy. We have that
\begin{equation*}
\medmath{\det \left(\begin{array}{ccc}
1& -\cos a_{14}& -\cos a_{13}\\
-\cos a_{14}& 1& -\cos a_{12}\\
-\cos a_{13}& -\cos a_{12}& 1
\end{array}\right) =}
\end{equation*}
\begin{equation*}
\medmath{= \det \left(\begin{array}{ccc}
1& -\cos a_{14}& -\cos a_{13}\\
0& \sin^{2} a_{14}& - {\mu^{\prime}}^{2}_{34}\, \sin a_{13} \sin a_{14}\\
0& - {\mu^{\prime}}^{2}_{34}\, \sin a_{13} \sin a_{14} & \sin^{2} a_{13}
\end{array}\right) =}
\end{equation*}
\begin{equation*}
= (1 - ({\mu^{\prime}}^{2}_{34})^{2}) \sin^{2} a_{13} \sin^{2} a_{14}
= \varepsilon_1\,(\mu^{2}_{34})^{2} \sin^{2} a_{13} \sin^{2} a_{14} = \varepsilon_1\,M^{2}_{234}.
\end{equation*}
By permuting the set $\{i,j,k,l\} = \{1,2,3,4\}$, one gets all other identities of the lemma.
\end{proof}

\begin{lemma}\label{lemma_detG}
For $\{i,j,k,l\} = \{1,2,3,4\}$, we have that
\begin{equation*}
- \det G = \sin^{2}a_{jk}\, \sin^{2}a_{jl}\, \sin^{2}a_{kl}\, (M^{i})^{2}.
\end{equation*}
\end{lemma}
\begin{proof}
Let us subsequently compute 
\begin{equation*}
\det G = \det \left( \begin{array}{cccc}
1& -\cos a_{34}& -\cos a_{24}& -\cos a_{23}\\
-\cos a_{34}& 1& -\cos a_{14}& -\cos a_{13}\\
-\cos a_{24}& -\cos a_{14}& 1& -\cos a_{12}\\
-\cos a_{23}& -\cos a_{13}& -\cos a_{12}& 1
\end{array} \right) = 
\end{equation*}
\begin{equation*}
\medmath{\det \left( \begin{array}{cccc}
1& -\cos a_{34}& -\cos a_{24}& -\cos a_{23}\\
0& \sin^{2} a_{34}& -{\mu^{\prime}}^{1}_{23} \sin a_{24} \sin a_{34}& -{\mu^{\prime}}^{1}_{24} \sin a_{23} \sin a_{34}\\
0& -{\mu^{\prime}}^{1}_{23} \sin a_{24} \sin a_{34}& \sin^{2} a_{24}& -{\mu^{\prime}}^{1}_{34} \sin a_{23} \sin a_{24}\\
0& -{\mu^{\prime}}^{1}_{24} \sin a_{23} \sin a_{34}& -{\mu^{\prime}}^{1}_{34} \sin a_{23} \sin a_{24}& \sin^{2} a_{23}
\end{array} \right) =} 
\end{equation*}
\begin{equation*}
\sin^{2} a_{23}\, \sin^{2} a_{24} \sin^{2} a_{34}\hspace*{0.1in} \det \left( \begin{array}{ccc}
1& -{\mu^{\prime}}^{1}_{23}& -{\mu^{\prime}}^{1}_{24}\\
-{\mu^{\prime}}^{1}_{23}& 1& -{\mu^{\prime}}^{1}_{34}\\
-{\mu^{\prime}}^{1}_{24}& -{\mu^{\prime}}^{1}_{34}& 1
\end{array} \right)=
\end{equation*}
\begin{equation*}
\sin^{2} a_{23}\, \sin^{2} a_{24} \sin^{2} a_{34}\hspace*{0.1in} \det \left( \begin{array}{ccc}
1& -{\mu^{\prime}}^{1}_{23}& -{\mu^{\prime}}^{1}_{24}\\
0& \varepsilon_4 (\mu^{1}_{23})^{2}& -\sigma^{\prime}_{34} \mu^{1}_{23} \mu^{1}_{24}\\
0& -\sigma^{\prime}_{34} \mu^{1}_{23} \mu^{1}_{24}& \varepsilon_3 (\mu^{1}_{24})^{2}
\end{array} \right)=
\end{equation*}
\begin{equation*}
\sin^{2} a_{23}\, \sin^{2} a_{24} \sin^{2} a_{34}\hspace*{0.1in} (\varepsilon_3 \varepsilon_4 - (\sigma^{\prime}_{34})^{2}) (\mu^{1}_{23} \mu^{1}_{24})^{2} = 
\end{equation*}
\begin{equation*}
- \sin^{2} a_{23}\, \sin^{2} a_{24} \sin^{2} a_{34}\hspace*{0.1in} (M^{1})^{2}.
\end{equation*}
Here we used the Cosine Law for links in the second equality and the first Cosine Law for faces in the fourth equality. Also, we used the fact that for $\{i,j,k,l\} = \{1,2,3,4\}$ one has $1 - \varepsilon_l (\mu^{i}_{jk})^{2} = ({\mu^{\prime}}^{i}_{jk})^{2}$ (in the third equality) and $\sigma^{2}_{ij} - (\sigma^{\prime}_{ij})^{2} = \varepsilon_i \varepsilon_j$ (in the sixth equality). All other identities of the lemma follow by permuting the set $\{i,j,k,l\}=\{1,2,3,4\}$.
\end{proof}

\section{Dual Jacobian of a generalised hyperbolic tetrahedron}\label{jacobian1}

In this section we shall compute the entries of the dual Jacobian matrix $\mathrm{Jac}^{\star}(T)$ of a generalised hyperbolic tetrahedron $T$.

\begin{theorem}\label{thm_jacobian1}
Let $T$ be a generalised hyperbolic tetrahedron. Then
\begin{equation*}
\mathrm{Jac}^{\star}(T) := \frac{\partial(\ell_{12}, \ell_{13}, \ell_{14}, \ell_{23}, \ell_{24}, \ell_{34})}{\partial(a_{12}, a_{13}, a_{14}, a_{23}, a_{24}, a_{34})} = -\eta\, \mathscr{D} \mathscr{S} \mathscr{D},
\end{equation*}
where
\begin{equation*}
\eta := \left( \frac{\Pi^4_{i=1} \varepsilon_i\, \det G_{ii}}{(-\det G)^3} \right)^{1/2},\,
\mathscr{D} := \left( \begin{array}{cccccc}
\sigma_{12}& & & & & \\
& \sigma_{13}& & & & \\
& & \sigma_{14}& & & \\
& & & \sigma_{23}& & \\
& & & & \sigma_{24}& \\
& & & & & \sigma_{34}
\end{array} \right)
\end{equation*}
and
\begin{equation*}
\mathscr{S} := \left( \begin{array}{cccccc}
\omega_{12}& \varepsilon_1 \sigma^{\prime}_{14}& \varepsilon_1 \sigma^{\prime}_{13}& \varepsilon_2 \sigma^{\prime}_{24}& \varepsilon_2 \sigma^{\prime}_{23}& 1\\
\varepsilon_1 \sigma^{\prime}_{14}& \omega_{13}& \varepsilon_1 \sigma^{\prime}_{12}& \varepsilon_3 \sigma^{\prime}_{34}& 1& \varepsilon_3 \sigma^{\prime}_{23}\\
\varepsilon_1 \sigma^{\prime}_{13}& \varepsilon_1 \sigma^{\prime}_{12}& \omega_{14}& 1& \varepsilon_4 \sigma^{\prime}_{34}& \varepsilon_4 \sigma^{\prime}_{24}\\
\varepsilon_2 \sigma^{\prime}_{24}& \varepsilon_3 \sigma^{\prime}_{34}& 1& \omega_{23}& \varepsilon_2 \sigma^{\prime}_{12}& \varepsilon_3 \sigma^{\prime}_{13}\\
\varepsilon_2 \sigma^{\prime}_{23}& 1& \varepsilon_4 \sigma^{\prime}_{34}& \varepsilon_2 \sigma^{\prime}_{12}& \omega_{24}& \varepsilon_4 \sigma^{\prime}_{14}\\
1& \varepsilon_3 \sigma^{\prime}_{23}& \varepsilon_4 \sigma^{\prime}_{24}& \varepsilon_3 \sigma^{\prime}_{13}& \varepsilon_4 \sigma^{\prime}_{14}& \omega_{34}
\end{array} \right),
\end{equation*}
where
\begin{equation*}
\omega_{kl} := \frac{\sigma^{\prime}_{ik}\sigma^{\prime}_{jl} + \varepsilon_l \sigma^{\prime}_{il}\sigma^{\prime}_{jl}\sigma^{\prime}_{kl} + \sigma^{\prime}_{il}\sigma^{\prime}_{jk} + \varepsilon_k \sigma^{\prime}_{ik}\sigma^{\prime}_{jk}\sigma^{\prime}_{kl}}{\sigma^2_{kl}}.
\end{equation*}
\end{theorem}
\begin{proof}
We compute the respective derivatives, that constitute the entries of $\mathrm{Jac}^{\star}(T)$. Suppose that $\varepsilon_i \neq 0$, $i\in\{1,2,3,4\}$, since the cases when $\varepsilon_j = 0$ for some $j\in\{1,2,3,4\}$ can be dealt with in an analogous manner. Then for $\{i,j,k,l\}=\{1,2,3,4\}$, one has
\begin{equation*}
\frac{\partial \ell_{kl}}{\partial a_{ij}} = \frac{\partial \ell_{kl}}{\partial b^{j}_{kl}}\, \frac{\partial b^{j}_{kl}}{\partial a_{ij}} = - \varepsilon_i\,\frac{\mu^{j}_{kl}}{M^{j}}\,\cdot\,\varepsilon_i\,\frac{\sin a_{ij}}{M_{jkl}} \overbrace{=}^{(1)} -\frac{1}{M^{j}}\,\frac{\sin a_{ij}}{\sin a_{ik} \sin a_{il}} = 
\end{equation*}
\begin{equation*}
\medmath{-\frac{1}{M^{j}}\,\frac{\sin a_{ij}}{\sin a_{ik} \sin a_{il}} \frac{1}{\sigma_{ij}} \frac{1}{\sigma_{kl}}\,\sigma_{ij} \sigma_{kl} \overbrace{=}^{(2)} -\frac{1}{M^{j}}\,\frac{\sin a_{ij}}{\sin a_{ik} \sin a_{il}}\,\sigma_{ij} \sigma_{kl}\,\frac{\mu^{i}_{jk} \mu^{i}_{jl}}{M^{i}}\,\frac{\mu^{l}_{ik}\mu^{l}_{jk}}{M^{l}} \overbrace{=}^{(3)}} 
\end{equation*}
\begin{equation*}
-\frac{M_{ijk} M_{ikl} M_{ijl} M_{jkl}}{\sqrt{(-\det G)^{3}}}\,\sigma_{ij} \sigma_{kl} \overbrace{=}^{(4)} -\sqrt{\frac{\Pi^{4}_{i=1} \varepsilon_i \det G_{ii}}{(-\det G)^{3}}}\,\sigma_{ij}\sigma_{kl} = -\eta\, \sigma_{ij} \sigma_{kl}.
\end{equation*}
Here we used the definitions of vertex and face momenta, as well as Lemmata \ref{lemma_l_b} and \ref{lemma_detG}. Indeed, in (1) we have that $M_{ijk} = \mu^{j}_{kl}\, \sin a_{ik}\, \sin a_{il}$ and in (2) we use the fact that $\sigma_{ij} = \frac{\mu^{l}_{ik}\, \mu^{l}_{jk}}{M^l}$, $\sigma_{kl} = \frac{\mu^{i}_{jk}\, \mu^{i}_{jl}}{M^{i}}$. In (3) we use
\begin{equation*}
\mu^{i}_{jk} = \frac{M_{ijk}}{\sin a_{jl}\, \sin a_{kl}},\hspace*{0.1in} \mu^{l}_{ik} = \frac{M_{ikl}}{\sin a_{ij}\, \sin a_{jk}},
\end{equation*}
\begin{equation*}
\mu^{i}_{jl} = \frac{M_{ijl}}{\sin a_{jk}\, \sin a_{kl}},\hspace*{0.1in} \mu^{l}_{jk} = \frac{M_{jkl}}{\sin a_{ij}\, \sin a_{ik}},
\end{equation*}
together with the identities of Lemma~\ref{lemma_detG}. In (4) we use Lemma~\ref{lemma_detGii}. 

Analogous to the above, we compute for $\{i,j,k,l\}=\{1,2,3,4\}$,
\begin{equation*}
\frac{\partial \ell_{kl}}{\partial a_{ik}} = \frac{\partial \ell_{kl}}{\partial b^{i}_{jk}}\, \frac{\partial b^{i}_{jk}}{\partial a_{ik}} = -\varepsilon_k \frac{\sin a_{ik}}{M_{ijl}}\,\frac{\mu^{i}_{kl}}{M^{i}}\,\sigma^{\prime}_{jk} = 
\end{equation*}
\begin{equation*}
-\varepsilon_k \frac{\sqrt{\varepsilon_j \det G_{jj}}}{M^{i} \sin a_{jk} \sin a_{jl}}\, \frac{\sin a_{ik}}{\sqrt{\varepsilon_k \det G_{kk}}}\, \frac{1}{\sigma_{ik} \sigma_{kl}}\, \sigma_{ik} \sigma_{kl}\, \sigma^{\prime}_{jk}=
\end{equation*}
\begin{equation*}
-\varepsilon_k \frac{\sqrt{\varepsilon_j \det G_{jj}}}{M^{i} \sin a_{jk} \sin a_{jl}}\, \frac{\sin a_{ik}}{\sqrt{\varepsilon_k \det G_{kk}}}\,\frac{\mu^{i}_{jk} \mu^{i}_{jl}}{M^{i}}\, \frac{\mu^{j}_{il} \mu^{j}_{kl}}{M^{j}}\, \sigma_{ik}\sigma_{kl}\, \sigma^{\prime}_{jk} =
\end{equation*}
\begin{equation*}
-\varepsilon_k \frac{\sqrt{\varepsilon_j \det G_{jj}}}{\sqrt{\varepsilon_k \det G_{kk}}} \frac{\sqrt{\varepsilon_l \det G_{ll}} \sqrt{\varepsilon_k \det G_{kk}} \sqrt{\varepsilon_k \det G_{kk}} \sqrt{\varepsilon_i \det G_{ii}}}{\sqrt{(-\det G)^3}}\,\sigma_{ik}\sigma_{kl}\,\sigma^{\prime}_{jk} =
\end{equation*}
\begin{equation*}
-\varepsilon_k \sqrt{\frac{\Pi^{4}_{i=1} \varepsilon_i \det G_{ii}}{(-\det G)^{3}}}\,\sigma_{ik}\sigma_{kl}\,\sigma^{\prime}_{jk} = -\varepsilon_k\, \eta\, \sigma_{ik} \sigma_{kl}\, \sigma^{\prime}_{jk}.
\end{equation*}

Finally, for $\{i,j,k,l\}=\{1,2,3,4\}$, we compute the derivative
\begin{equation*}
\frac{\partial \ell_{kl}}{\partial a_{kl}} = \frac{\partial \ell_{kl}}{\partial b^{i}_{jk}}\, \frac{\partial b^{i}_{jk}}{\partial a_{kl}} + \frac{\partial \ell_{kl}}{\partial b^{i}_{jl}}\, \frac{\partial b^{i}_{jl}}{\partial a_{kl}}.
\end{equation*}
Since the two terms of the above sum are symmetric under the permutation of $k$ and $l$, we may compute only the first one. The second one will be analogous. By Lemmata \ref{lemma_l_b} and \ref{lemma_b_a}, we get
\begin{equation*}
\frac{\partial \ell_{kl}}{\partial b^{i}_{jk}}\, \frac{\partial b^{i}_{jk}}{\partial a_{kl}} = -\varepsilon_l \frac{\mu^{i}_{kl}}{M^{i}}\, \frac{\sin a_{il}}{M_{ijk}}\,\sigma^{\prime}_{jl}\, {\mu^{\prime}}^{j}_{ik} \overbrace{=}^{(5)} 
\end{equation*}
\begin{equation*}
-(\sigma^{\prime}_{ik}\sigma^{\prime}_{jl} + \varepsilon_l \sigma^{\prime}_{il}\sigma^{\prime}_{jl}\sigma^{\prime}_{kl})\,\frac{\mu^{i}_{kl}}{M^{i}}\, \frac{\sin a_{il}}{M_{ijk}}\,\frac{1}{\sigma_{il}\sigma_{kl}} \overbrace{=}^{(6)} 
\end{equation*}
\begin{equation*}
-(\sigma^{\prime}_{ik}\sigma^{\prime}_{jl} + \varepsilon_l \sigma^{\prime}_{il}\sigma^{\prime}_{jl}\sigma^{\prime}_{kl})\,\frac{\sqrt{\varepsilon_j \det G_{jj}}}{M^{i} \sin a_{jk} \sin a_{jl}}\,\frac{\sin a_{il}}{\sqrt{\varepsilon_l \det G_{ll}}}\, \frac{\mu^{i}_{jk} \mu^{i}_{jl}}{M^{i}} \frac{\mu^{j}_{ik} \mu^{j}_{kl}}{M^{j}} \overbrace{=}^{(7)}
\end{equation*}
\begin{equation*}
\medmath{-(\sigma^{\prime}_{ik}\sigma^{\prime}_{jl} + \varepsilon_l \sigma^{\prime}_{il}\sigma^{\prime}_{jl}\sigma^{\prime}_{kl})\,\sqrt{\frac{\varepsilon_j \det G_{jj}}{(-\det G)^3}}\, \frac{\sqrt{\varepsilon_l \det G_{ll}} \sqrt{\varepsilon_k \det G_{kk}} \sqrt{\varepsilon_l \det G_{ll}} \sqrt{\varepsilon_i \det G_{ii}}}{\sqrt{\varepsilon_l \det G_{ll}}} =}
\end{equation*}
\begin{equation*}
-(\sigma^{\prime}_{ik}\sigma^{\prime}_{jl} + \varepsilon_l \sigma^{\prime}_{il}\sigma^{\prime}_{jl}\sigma^{\prime}_{kl})\,\sqrt{\frac{\Pi^{4}_{i=1} \varepsilon_i \det G_{ii}}{(-\det G)^{3}}} = -\eta\, (\sigma^{\prime}_{ik}\sigma^{\prime}_{jl} + \varepsilon_l \sigma^{\prime}_{il}\sigma^{\prime}_{jl}\sigma^{\prime}_{kl}).
\end{equation*}
Here, in (5) we used the second Cosine Law for faces and in (6) we used the equality $M_{ikl} = \mu^{i}_{kl}\, \sin a_{jk}\, \sin a_{jl}$ together with Lemma~\ref{lemma_detGii}. In (7) we perform a computation analogous to (3). 

Thus, we obtain 
\begin{equation*}
\frac{\partial \ell_{kl}}{\partial a_{kl}} = \frac{\partial \ell_{kl}}{\partial b^{i}_{jk}}\, \frac{\partial b^{i}_{jk}}{\partial a_{kl}} + \frac{\partial \ell_{kl}}{\partial b^{i}_{jl}}\, \frac{\partial b^{i}_{jl}}{\partial a_{kl}} = 
\end{equation*}
\begin{equation*}
-\eta\, (\sigma^{\prime}_{ik}\sigma^{\prime}_{jl} + \varepsilon_l \sigma^{\prime}_{il}\sigma^{\prime}_{jl}\sigma^{\prime}_{kl})
-\eta\, (\sigma^{\prime}_{il}\sigma^{\prime}_{jk} + \varepsilon_k \sigma^{\prime}_{ik}\sigma^{\prime}_{jk}\sigma^{\prime}_{kl}) = -\eta\,\omega_{kl}\,\sigma^{2}_{kl}.
\end{equation*}
The proof is completed. 
\end{proof}

\section{Dual Jacobian of a doubly truncated hyperbolic tetrahedron}\label{jacobian2}

Let us consider the case when $T$ is a (mildly) doubly truncated tetrahedron depicted in Fig.~\ref{fig_tetr_doubly_trunc} with dihedral angles $\theta_i$ and edge lengths $\ell_i$, $i\in \{1,2,3,4,5,6\}$. We suppose that the vertices cut off by the respective polar planes are $\mathrm{v}_1$ and $\mathrm{v}_2$. 

\begin{figure}[ht]
\begin{center}
\includegraphics* [totalheight=5cm]{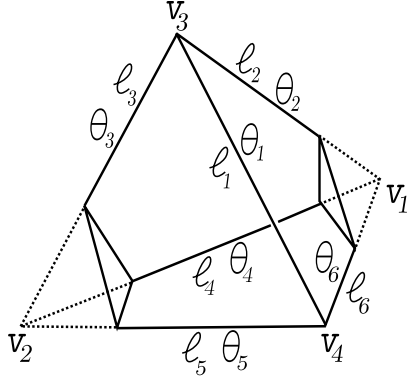}
\end{center}
\caption{Doubly truncated tetrahedron (mild truncation)} \label{fig_tetr_doubly_trunc}
\end{figure}

If $T$ is mildly truncated then the formula from Theorem~\ref{thm_jacobian1} applies. If $T$ is a prism truncated tetrahedron, as in Fig.~\ref{fig_tetr_prism_trunc}, with dihedral angles $\mu$, $\theta_i$ and edge lengths $\ell$, $\ell_i$, $i\in \{1,2,3,5,6\}$ then its Gram matrix is given by 
\begin{equation*}
G = \left( \begin{array}{cccc}
1 & -\cos \theta_1 & -\cos \theta_5 & -\cos \theta_3\\
-\cos \theta_1 & 1 & -\cos \theta_6 & -\cos \theta_2\\
-\cos \theta_5 & -\cos \theta_6 & 1 & -\cosh \ell\\
-\cos \theta_3 & -\cos \theta_2 & -\cosh \ell & 1
\end{array} \right),
\end{equation*}
which is a slightly different notation compared to \cite{KM2012, KM-err}. 

Each link $L(\mathrm{v}_k)$, $k=1,2$, is a hyperbolic quadrilateral with two right same-side angles, which can be seen as a hyperbolic triangle with a single truncated vertex. Each link $L(\mathrm{v}_k)$, $k=3,4$, is a spherical triangle. In the definitions of Section~\ref{preliminaries} we change each $b^{i}_{1j}$, with $i, j \in \{2,3,4\}$, $i\neq j$, for $b^{i}_{1j} + \sqrt{-1} \frac{\pi}{2}$ and each $b^{i}_{2j}$, with $i, j \in \{1,3,4\}$, $i\neq j$, for $b^{i}_{2j} + \sqrt{-1} \frac{\pi}{2}$. Thus, some of the vertex and face momenta become complex numbers. All the trigonometric rules of Section~\ref{preliminaries} still hold grace to \cite[Section~4.3]{Cho}. Computing the respective derivatives in a complete analogy to the proof of Theorem~\ref{thm_jacobian1}, we obtain the following statement.

\begin{figure}[ht]
\begin{center}
\includegraphics* [totalheight=5cm]{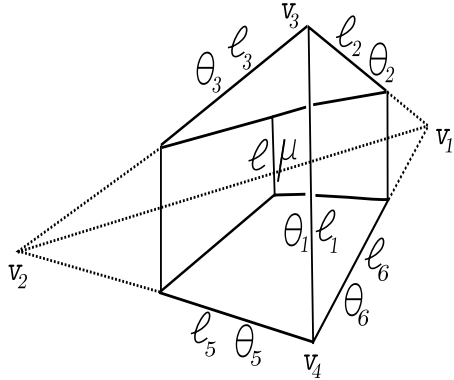}
\end{center}
\caption{Doubly truncated tetrahedron (prism truncation)} \label{fig_tetr_prism_trunc}
\end{figure}

\begin{theorem}\label{thm_jacobian2}
Let $T$ be a prism truncated tetrahedron depicted in Fig.~\ref{fig_tetr_prism_trunc}. Then by means of the analytic continuation $a_{12} := \sqrt{-1}\,\ell$, $\ell_{12} = \sqrt{-1}\,\mu$ we have
\begin{equation*}
\mathrm{Jac}^{\star}(T) := \frac{\partial (\mu, \ell_1, \ell_2, \ell_3, \ell_5, \ell_6)}{\partial (\ell, \theta_1, \theta_2, \theta_3, \theta_5, \theta_6)} = \frac{\partial (\ell_{12}, \ell_{34}, \ell_{13}, \ell_{23}, \ell_{24}, \ell_{14})}{\partial (a_{12}, a_{34}, a_{13}, a_{23}, a_{24}, a_{14})}.
\end{equation*}
\end{theorem}

\section{Volume of a hyperbolic prism}

Let $\vec{\alpha}_n$ denote the $n$-tuple $(\alpha_1, \dots, \alpha_n)$ with $0 < \alpha_k < \pi$, $k=1,\dots,n$. Let $\vec{\beta}_n$ and $\vec{\gamma}_n$ be analogous $n$-tuples. Let $\Pi_n := \Pi_n(\vec{\alpha}_n, \vec{\beta}_n, \vec{\gamma}_n)$ be the hyperbolic $n$-sided prism depicted in Fig.~\ref{fig_prism_n}, with the respective dihedral angles, as shown in the picture. 

\begin{figure}[ht]
\begin{center}
\includegraphics* [totalheight=6cm]{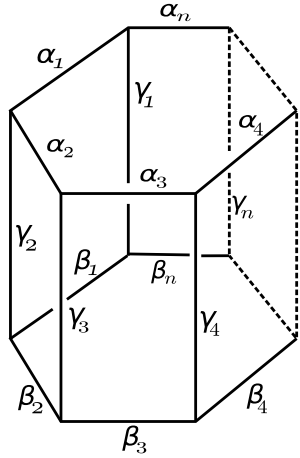}
\end{center}
\caption{The prism $\Pi_n(\vec{\alpha}_n, \vec{\beta}_n, \vec{\gamma}_n)$} \label{fig_prism_n}
\end{figure}

Let $S_k$, $k=1,\dots, n$, be the supporting hyperplane for the $k$-th side face of the prism $\Pi_n$ (we start numbering the faces anti-clockwise from the side face adjacent to the angles $\alpha_1$, $\beta_1$ and $\gamma_1$, $\gamma_2$), and let $S_0$ and $S_{n+1}$ be those of the top and the bottom face, correspondingly. For each $S_k$, $k=0,\dots, n+1$, let $S^{+}_k$ be the respective half-space containing the unit outer normal to it. Let $S^{-}_k = \mathbb{H}^3 \setminus S^{+}_k$. Then $\Pi_n = \bigcap^{n+1}_{i=0} S^{-}_i$. 

Let $T := T(\alpha, \alpha^{\prime}, \beta, \beta^{\prime}, \gamma; \ell)$ be the prism truncated tetrahedron depicted in Fig.~\ref{fig_tetr_prism_trunc2}. Here $\alpha$, $\alpha^{\prime}$, $\beta$, $\beta^{\prime}$ and $\gamma$ are the respective dihedral angles, $\ell$ is the length of the respective edge. The volume $\mathrm{Vol}\,T$ of the tetrahedron $T$ is given by \cite[Theorem~1]{KM2012}\footnote{in Section~\ref{volume-modified} we give a simplified formula for the volume of $T$.}. Let $v(\alpha, \alpha^{\prime}, \beta, \beta^{\prime}, \gamma; \ell) := \mathrm{Vol}\,T(\alpha, \alpha^{\prime}, \beta, \beta^{\prime}, \gamma; \ell)$ denote the respective volume function. 

\begin{figure}[ht]
\begin{center}
\includegraphics* [totalheight=6cm]{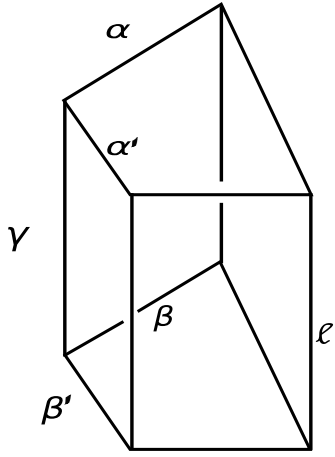}
\end{center}
\caption{The prism truncated tetrahedron $T(\alpha, \alpha^{\prime}, \beta, \beta^{\prime}, \gamma; \ell)$} \label{fig_tetr_prism_trunc2}
\end{figure}

Let $p_{0}p_{n+1}$ be the common perpendicular to $S_0$ and $S_{n+1}$. Let also define $k\oplus m := (k+m)\, \mathrm{mod}\, n$, for $k,m \in \mathbb{N}$. Then we can state the main theorem of this section.

\begin{theorem}\label{thm_prism_volume}
Let $\Pi_n = \Pi_n(\vec{\alpha}_n, \vec{\beta}_n, \vec{\gamma}_n)$ be a hyperbolic $n$-sided prism, as in Fig.~\ref{fig_prism_n}. If $p_{0}p_{n+1} \subset \Pi_n$, then the volume of $\Pi_n$ is given by the formula
\begin{equation*}
\mathrm{Vol}\,\Pi_n = \sum^{n}_{k=1} v(\alpha_{k}, \alpha_{k\oplus 1}, \beta_{k}, \beta_{k\oplus 1}, \gamma_{k\oplus 1}; \ell^{\star}),
\end{equation*}
where $\ell^{\star}$ is the unique solution to the equation $\frac{\partial \Phi}{\partial \ell}(\ell) = 0$, with
\begin{equation*}
\hspace*{0.1in} \Phi(\ell) := \pi \ell + \sum^{n}_{k=1} v(\alpha_{k}, \alpha_{k\oplus 1}, \beta_{k}, \beta_{k\oplus 1}, \gamma_{k\oplus 1}; \ell).
\end{equation*}
\end{theorem}

Let $P_k$, $k=1,\dots,n$, be the plane containing $p_0p_{n+1}$ and orthogonal to $S_k$. First, we consider the case when $p_0p_{n+1}$ lies inside the prism $\Pi_n$ and the planes $P_k$, $k=1,\dots, n$, divide the prism $\Pi_n$ into $n$ prism truncated tetrahedra, as shown in Fig.~\ref{fig_prism_subdivision1}. 

\begin{figure}[ht]
\begin{center}
\includegraphics* [totalheight=6cm]{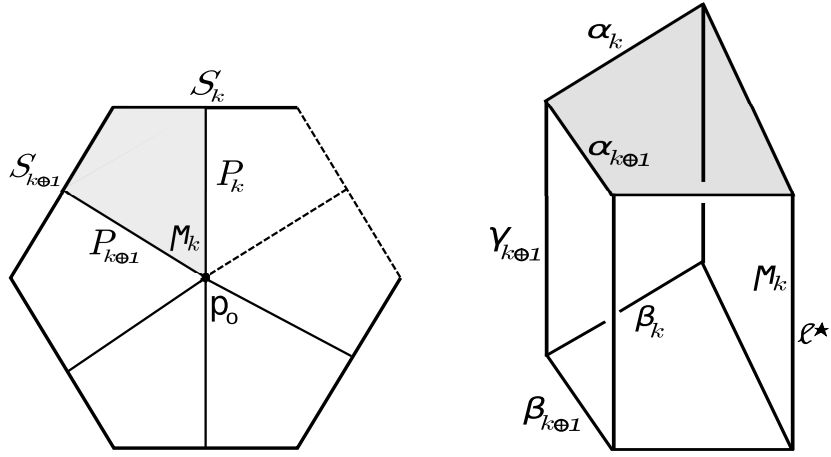}
\end{center}
\caption{The decomposition of $\Pi_n$ (top view, on the left) and the prism truncated tetrahedron $T_k$ (on the right)} \label{fig_prism_subdivision1}
\end{figure}

Then each $P_k$ meets the $k$-th side face of the prism $\Pi_n$. Thus, the planes $S_0$, $S_k$, $S_{k\oplus 1}$ and $S_{n+1}$ together with $P_k$ and $P_{k\oplus 1}$ become the supporting planes for the faces of a prism truncated tetrahedron, which we denote by $T_k$. Each $P_k$ is orthogonal to $S_k$, $S_{0}$ and $S_{n+1}$. The dihedral angles of $T_k$ inherited from the prism $\Pi_n$ are easily identifiable. Let $\mu_k$ denote the dihedral angle along the edge $p_0p_{n+1}$ and let $\ell^{\star}$ be its length. Then we have $T_k = T(\alpha_{k}, \alpha_{k\oplus 1}, \beta_{k}, \beta_{k\oplus 1}, \gamma_{k\oplus 1}; \ell^{\star})$, $k=1,\dots,n$. Clearly,
\begin{equation*}
\mathrm{Vol}\,\Pi_n = \sum^{n}_{k=1} \mathrm{Vol}\,T_k = \sum^{n}_{k=1} v(\alpha_{k}, \alpha_{k\oplus 1}, \beta_{k}, \beta_{k\oplus 1}, \gamma_{k\oplus 1}; \ell^{\star}).
\end{equation*}
Thus, we have to prove only the following statement.

\begin{proposition}\label{prop_prism_volume1}
If the common perpendicular $p_0p_{n+1}$ is inside the prism $\Pi_n$ and each $P_k$ meets the respective side also inside $\Pi_n$, $k=1,\dots,n$, then the equation $\frac{\partial \Phi}{\partial \ell} = 0$ has a unique solution $\ell = \ell^{\star}$, the length of $p_0p_{n+1}$.
\end{proposition}
\begin{proof}
Let us consider the collection of prism truncated tetrahedra $T_k = T(\alpha_k, \alpha_{k\oplus 1}, \beta_k, \beta_{k\oplus 1}, \gamma_{k\oplus 1}; \ell)$, $k=1,\dots,n$. Each pair $\{T_{k}, T_{k\oplus 1}\}$ of them has an isometric face corresponding to the plane $P_{k\oplus 1}$. Indeed, each such face is completely determined by the plane angles (two right angles at the side of length $\ell$, the angles $\alpha_k$ and $\beta_{k}$ at the opposite side) and one side length. We obtain the prism $\Pi_n(\vec{\alpha}_n, \vec{\beta}_n, \vec{\gamma}_n)$ by glueing the tetrahedra $T_k$ together along the faces $P_k$, $k=1,\dots,n$, in the respective order. Their edges of length $\ell$ match together, and one obtains a prism if the angle sum of the dihedral angles $\mu_k$, $k=1,\dots,n$, along them equals $2\pi$. We have that
\begin{equation*}
\frac{\partial \Phi}{\partial \ell} = \pi + \sum^{n}_{k=1} \frac{\partial v}{\partial \ell}(\alpha_{k}, \alpha_{k\oplus 1}, \beta_{k}, \beta_{k\oplus 1}, \gamma_{k\oplus 1}; \ell).
\end{equation*}
Since $v$ is the volume function from \cite[Theorem~1]{KM2012}, then by applying the Schl\"{a}fli formula \cite[Equation~1]{Milnor} one obtains 
\begin{equation*}
\frac{\partial \Phi}{\partial \ell} = \pi - \frac{1}{2}\sum^{n}_{k=1} \mu_k.
\end{equation*}
Thus, whenever the tetrahedra $T_k$ constitute a prism, we have $\sum^{n}_{k=1} \mu_k = 2\pi$ or, equivalently, $\frac{\partial \Phi}{\partial \ell} = 0$. The length $\ell$ in this case is exactly the length of the common perpendicular $p_0p_{n+1}$ to the planes $S_0$ and $S_{n+1}$. 

The rest is to prove that $\ell = \ell^{\star}$ is a unique solution. In order to do so, we shall show that $\frac{\partial \mu_k}{\partial \ell} > 0$, $k=1,\dots,n$. By using Theorem~\ref{thm_jacobian2} we get the following formulae for a prism truncated tetrahedron (as depicted in Fig.~\ref{fig_tetr_prism_trunc}): 
\begin{equation*}
\frac{\partial \ell_2}{\partial \ell} = - \eta \sin\mu_k \sinh\ell_6 \cosh\ell_2,\hspace*{0.1in}\frac{\partial \ell_3}{\partial \ell} = - \eta \sin\mu_k \sinh\ell_5 \cosh\ell_3,
\end{equation*}
\begin{equation*}
\frac{\partial \ell_5}{\partial \ell} = - \eta \sin\mu_k \sinh\ell_3 \cosh\ell_5,\hspace*{0.1in}\frac{\partial \ell_6}{\partial \ell} = - \eta \sin\mu_k \sinh\ell_2 \cosh\ell_6.
\end{equation*}
Note that the above derivatives are all negative. In our present notation it means that for each prism truncated tetrahedron $T_k$, $k=1,\dots,n$, the edges of the top and bottom faces inherited from the prism $\Pi_n$ diminish their length if we increase solely the parameter $\ell$. Recall that $T_k = T(\alpha_k, \alpha_{k\oplus 1}, \beta_k, \beta_{k\oplus 1}, \gamma_{k\oplus 1}; \ell)$, and let us denote $T^{\prime}_k := T(\alpha_k, \alpha_{k\oplus 1}, \beta_k, \beta_{k\oplus 1}, \gamma_{k\oplus 1}; \ell^{\prime})$ with $\ell^{\prime} > \ell$.  

\begin{figure}[ht]
\begin{center}
\includegraphics* [totalheight=7cm]{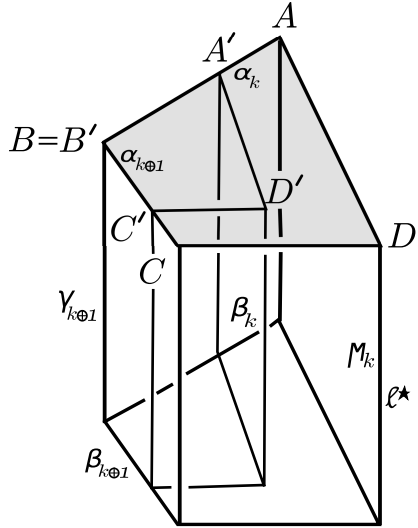}
\end{center}
\caption{Prisms $T_k$ and $T^{\prime}_k$ with top faces marked} \label{abcd_prisms}
\end{figure}

Let $ABCD$ be the top (equiv., bottom) face of $T_k$, as shown in Fig.~\ref{abcd_prisms}, and $A^{\prime}B^{\prime}C^{\prime}D^{\prime}$ be the top (equiv., bottom) face of $T^{\prime}_k$. Since the dihedral angles accept for $\mu_k$ and $\mu^{\prime}_k$ remain the same, the plane angles of $ABCD$ at $A$, $B$, $C$ and those of $A^{\prime}B^{\prime}C^{\prime}D^{\prime}$ at $A^{\prime}$, $B^{\prime}$ and $C^{\prime}$ are respectively equal. One sees easily that we can match then $ABCD$ and $A^{\prime}B^{\prime}C^{\prime}D^{\prime}$ such that $B$ and $B^{\prime}$ coincide, the sides $AB$ and $A^{\prime}B^{\prime}$, $BC$ and $B^{\prime}C^{\prime}$ overlap and the point $D^{\prime}$ lies inside the quadrilateral $ABCD$. Then the area of $A^{\prime}B^{\prime}C^{\prime}D^{\prime}$ is less than that of $ABCD$. Equivalently, by the angle defect formula \cite[Theorem~1.1.7]{Buser}, $\mu^{\prime}_k > \mu_k$. Thus, $\frac{\partial \mu_k}{\partial \ell} > 0$, $k=1,\dots, n$, and the proposition follows.
\end{proof}

However, there is a possibility that, although the common perpendicular $p_0p_{n+1}$ is entirely inside the prism $\Pi_n$, one (or several) of the planes $P_k$ meets the respective $S_k$ partially outside of the face $S_k$. 

First we consider the case when a single plane $P_k$ meets $S_k$ entirely outside, as depicted in Fig.~\ref{fig_prism_subdivision2}. Like this, we obtain the figure shaded in grey, that consists of two triangular prisms sharing an edge. 

\begin{figure}[ht]
\begin{center}
\includegraphics* [totalheight=6cm]{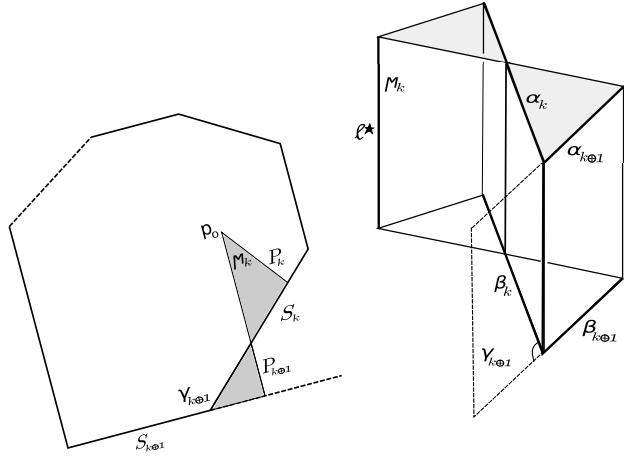}
\end{center}
\caption{The decomposition of $\Pi_n$ (top view, on the left) and the ``butterfly'' prism truncated tetrahedron $T_k$ (on the right)} \label{fig_prism_subdivision2}
\end{figure}

Second we consider the case when a single plane $P_k$ meets $S_k$ partially outside, as depicted in Fig.~\ref{fig_prism_subdivision3}. Like this, we obtain a more complicated figure that consists of two tetrahedra sharing an edge (one of which has two truncated vertices).

\begin{figure}[ht]
\begin{center}
\includegraphics* [totalheight=6cm]{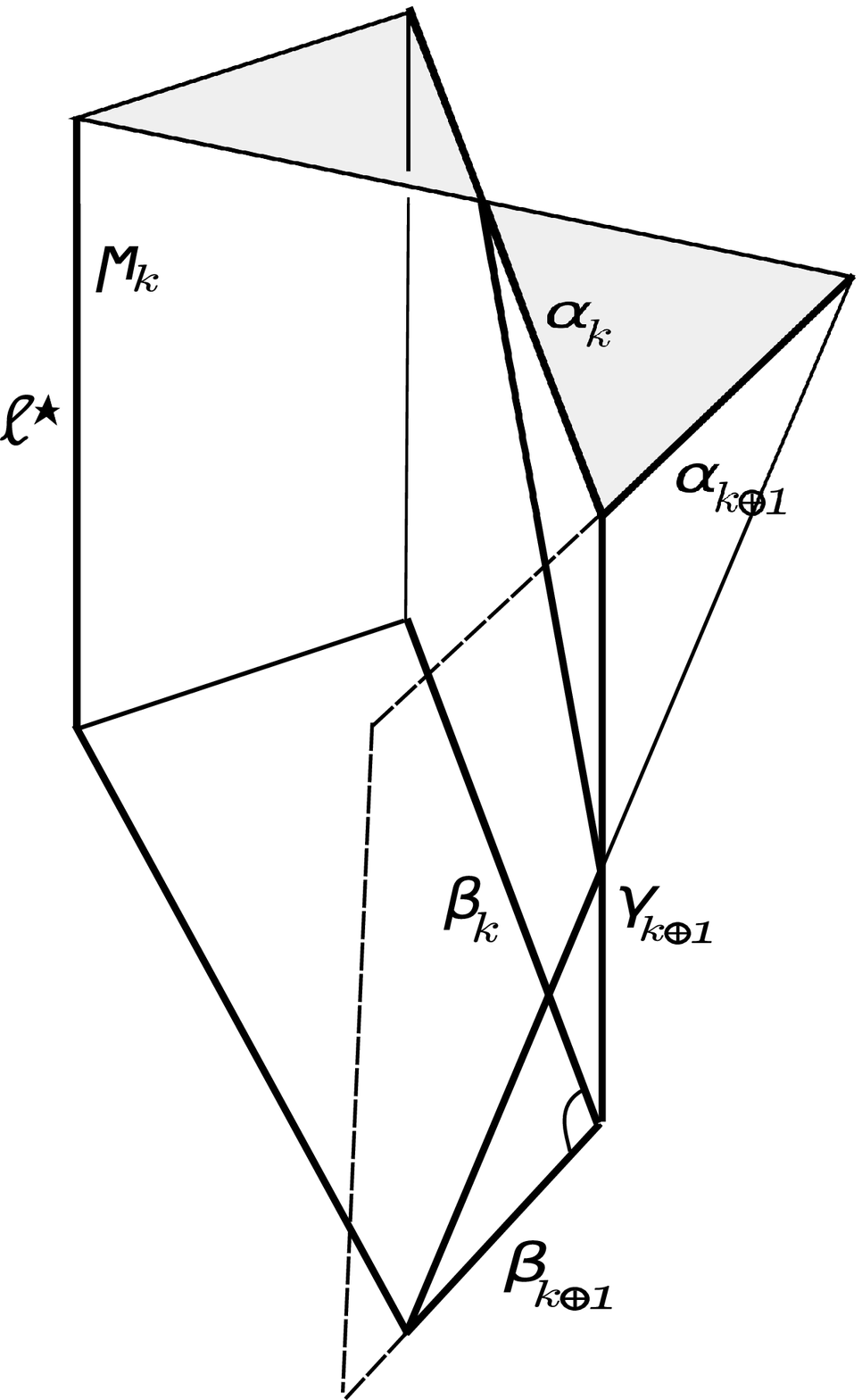}
\end{center}
\caption{Another ``butterfly'' prism truncated tetrahedron $T_k$} \label{fig_prism_subdivision3}
\end{figure}

Thus the planes $S_0$, $P_{k}$, $P_{k\oplus 1}$, $S_k$, $S_{k\oplus 1}$ and $S_{n+1}$ bound a ``butterfly'' prism. We put $k=1$, for clarity. In the general case, $k \geq 2$, one uses induction on the number of planes $P_k$ meeting $S_k$ outside of $\Pi_n$. Here, some other cases of ``butterfly'' prisms are possible.

\begin{proposition}\label{prop_prism_volume2}
If the common perpendicular $p_0p_{n+1}$ is completely inside the prism $\Pi_n$, the plane $P_1$ meets the plane $S_1$ outside of $\Pi_n$, and all other $P_k$, $k=2,\dots,n$, meet the respective side faces inside $\Pi_n$, then the volume of the prism equals
\begin{equation*}
\mathrm{Vol}\,\Pi_n = \sum^{n}_{k=1} v(\alpha_{k}, \alpha_{k\oplus 1}, \beta_{k}, \beta_{k\oplus 1}, \gamma_{k\oplus 1}; \ell^{\star}),
\end{equation*}
where $\ell^{\star}$ is the unique solution to the equation $\frac{\partial \Phi}{\partial \ell}(\ell) = 0$, with
\begin{equation*}
\hspace*{0.1in} \Phi(\ell) := \pi \ell + \sum^{n}_{k=1} v(\alpha_{k}, \alpha_{k\oplus 1}, \beta_{k}, \beta_{k\oplus 1}, \gamma_{k}; \ell).
\end{equation*}
\end{proposition}

\begin{proof}
We start with the case of a ``butterfly'' prism depicted in Fig.~\ref{fig_prism_subdivision2}. Let us observe that the ``butterfly'' prism $T_1$ overlaps with the subsequent prism truncated tetrahedron $T_2$ exactly on its part $T^{(o)}_1$ outside of $\Pi_n$. The part of $T_1$ inside $\Pi_n$, called $T^{(i)}_1$, contributes to the total volume of the prism. The volume of $T^{(o)}_1$ is excessive in the respective volume formula and should be subtracted. In fact, we prove that
\begin{equation*}
v(\alpha_1,\alpha_2,\beta_1,\beta_2,\gamma_2; \ell^{\star}) = V:= \mathrm{Vol}\,T^{(i)}_1 - \mathrm{Vol}\,T^{(o)}_1,
\end{equation*}
which implies that the excess in volume brought by $T_2$ is eliminated by the term ``$- \mathrm{Vol}\,T^{(o)}_1$''.

In order to do so, let us denote by $\theta$ the dihedral angle along the common edge of the triangular prisms $T^{(o)}_1$ and $T^{(i)}_1$. Let $\ell_{\theta}$ be the length of this edge. Let $\gamma := \gamma_2$ and let $\ell_{\gamma}$ be the length of the vertical edge with dihedral angle $\gamma$. We know that $\frac{\partial V}{\partial \gamma} = -\frac{1}{2}\,\ell_{\gamma}$, by the structure of the volume formula for a prism truncated tetrahedron. Indeed, the function $V$ does not correspond to the volume of a real prism truncated tetrahedron any more, however all the metric relations defining the dihedral angles between the respective planes are preserved. Thus, after computing the derivative $\frac{\partial V}{\partial \ell}$ analogous to \cite{KM2012}, we obtain the latter equality. Now we compute the respective derivatives for the parts of the ``butterfly'' prism $T_1$. 

\begin{figure}[ht]
\begin{center}
\includegraphics* [totalheight=6cm]{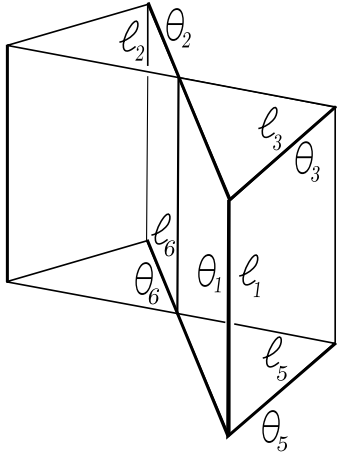}
\end{center}
\caption{Parametrising the ``butterfly'' prism depicted in Fig.~\ref{fig_prism_subdivision2}} \label{fig_butterfly_tetr}
\end{figure}

Observe that the parameter $\theta$ depends on $\gamma$, while we vary $\gamma$ and keep all other dihedral angles fixed. Let us denote $\hat{\gamma} = \pi - \gamma$ for brevity. We have that
\begin{equation*}
\frac{\partial \mathrm{Vol}\,T^{(o)}_1}{\partial \hat{\gamma}} = -\frac{\ell_{\gamma}}{2} -\frac{\ell_{\theta}}{2}\,\frac{\partial \theta}{\partial \hat{\gamma}}
\end{equation*}
and
\begin{equation*}
\frac{\partial \mathrm{Vol}\,T^{(i)}_{1}}{\partial \gamma} = -\frac{\ell_{\theta}}{2}\, \frac{\partial \theta}{\partial \gamma},
\end{equation*}
by the Schl\"{a}fli formula \cite[Equation~1]{Milnor}.

The above identities together with the fact that $\frac{\partial}{\partial \hat{\gamma}} = - \frac{\partial}{\partial \gamma}$ imply that
\begin{equation*}
\frac{\partial}{\partial \gamma_2}v(\alpha_1,\alpha_2,\beta_1,\beta_2,\gamma_2; \ell^{\star}) = \frac{\partial V}{\partial \gamma_2}.
\end{equation*}

By analogy, we can prove that 
\begin{equation*}
\frac{\partial}{\partial \xi}v(\alpha_1,\alpha_2,\beta_1,\beta_2,\gamma_2; \ell^{\star}) = \frac{\partial V}{\partial \xi},
\end{equation*}
for any $\xi \in \{ \alpha_1, \alpha_{n}, \beta_1, \beta_n, \mu_1 \}$. The volume formula for a prism truncated tetrahedron implies that by setting $\alpha_1 = \alpha_n = \pi/2$ and $\beta_1 = \beta_n = \pi/2$ we get $v(\alpha_1,\alpha_n,\beta_1,\beta_n,\gamma_2; \ell^{\star}) = 0$. In the case of a ``butterfly'' prism $T_1$, under the same assignment of dihedral angles, we have that the bases of the two triangular prisms become orthogonal to their lateral sides. Thus $T^{(i)}_1$ and $T^{(o)}_1$ degenerate into Euclidean prisms, which means that their volumes tend to zero. Thus, we obtain the identity $v(\alpha_1,\alpha_n,\beta_1,\beta_n,\gamma_2; \ell^{\star}) = V$. 

The proof of the monotonicity for the function $\frac{\partial\Phi}{\partial \ell}(\ell)$ is analogous to that in Proposition~\ref{prop_prism_volume1}. However, since the part $T^{(o)}_1$ contributes to the function $v(\alpha_1,\alpha_n,\beta_1,\beta_n,\gamma_2; \ell)$ with the negative sign, we have to replace the edge lengths $\ell_3$ and $\ell_5$ with $-\ell_3$ and $-\ell_5$, respectively, as shown in Fig.~\ref{fig_butterfly_tetr}. Then we recompute the respective derivatives of the lengths of the horizontal edges according to Theorem~\ref{thm_jacobian2}. We obtain that the lengths $\ell_2$ and $\ell_6$ diminish, as before, while the lengths $\ell_3$ and $\ell_5$ increase. This implies that the upper (resp., lower) triangular base of $T^{(i) \prime}_1$ can be placed entirely inside the upper (resp. lower) triangular base of $T^{(i)}_1$. By the area comparison argument, we have that $\mu^{\prime}_1 > \mu_1$. The inequality $\frac{\partial \mu_1}{\partial \ell} > 0$ follows. 

All other cases of ``butterfly prisms'' (e.g. that in Fig.~\ref{fig_prism_subdivision3}) can be considered by analogy.
\end{proof}

\medskip

\begin{remark}
In the general case, when the common perpendicular $p_{0}p_{n+1}$ does not lie entirely inside the prism $\Pi_n$, we expect that an analogue to Theorem~\ref{thm_prism_volume} holds with an exception that the equation $\frac{\partial \Phi}{\partial \ell}(\ell) = 0$ may have several solutions. However, one of these solutions is geometric and yields the volume of $\Pi_n$.
\end{remark}

\section{Modified volume formula}\label{volume-modified}

We modify the volume formula for a prism truncated tetrahedron from \cite{KM2012}, in order to reduce it to a simpler form. Indeed, the formula in \cite[Theorem~1]{KM2012} uses analytic continuation and accounts for possible branching with respect to any variable $a_j = e^{\ell}$, with some $j\in\{1,2,\dots,6\}$, and $a_k = e^{i\,\theta_k}$, for any $k\in\{1,2,\dots,6\}\setminus \{j\}$. Usually, we put $j=4$ for simplicity. However, the formula allows for intense truncation at any edge, since it is invariant under a permutation of the variables $a_l$, $l \in \{1,\dots,6\}$.

In our case, given a prism $\Pi_n$ and its decomposition into prism truncated tetrahedra $T_i$, $i\in\{1,\dots,n\}$, we know that only the common perpendicular $p_{0}p_{n+1}$ is produced by an intense truncation. Thus, we can always put $j=4$ and, moreover, the variable $a_4$ will be the only one that might cause branching. 

In this case, we suggest a simplified version of the formula from \cite[Theorem 1]{KM2012}. This formula also has less numeric discrepancies and performs faster, if used for an actual computation.

Let us put $a_k := e^{i\, \theta_k}$, $k\in \{1,2,3,5,6\}$, $a_4 := e^{\ell}$, and let $\mathscr{U} = \mathscr{U}(a_1,a_2,a_3,a_4,a_5,a_6,z)$ denote
\begin{eqnarray*}
\nonumber
\lefteqn{\mathscr{U} := \mathrm{Li}_2(z) + \mathrm{Li}_2(a_1a_2a_4a_5z) + \mathrm{Li}_2(a_1a_3a_4a_6z) + \mathrm{Li}_2(a_2a_3a_5a_6z)}\\
\nonumber && - \mathrm{Li}_2(-a_1a_2a_3z) - \mathrm{Li}_2(-a_1a_5a_6z) - \mathrm{Li}_2(-a_2a_4a_6z) - \mathrm{Li}_2(-a_3a_4a_5z),
\end{eqnarray*}
where $\mathrm{Li}_2(\circ)$ is the dilogarithm function.

Let $z_{-}$ and $z_{+}$ be two solutions to the equation $e^{z \frac{\partial \mathscr{U}}{\partial z}} = 1$ in the variable $z$. According to \cite{KM2012, MY}, these are
\begin{equation*}
z_{-} := \frac{-q_1-\sqrt{q^2_1-4q_0q_2}}{2q_2} \,\,\,\mbox{ and }\,\,\, z_{+} := \frac{-q_1+\sqrt{q^2_1-4q_0q_2}}{2q_2},
\end{equation*}
where
\begin{equation*}
q_0 := 1 + a_1a_2a_3 + a_1a_5a_6 + a_2a_4a_6 + a_3a_4a_5
+ a_1a_2a_4a_5 + a_1a_3a_4a_6 + a_2a_3a_5a_6,
\end{equation*}
\begin{eqnarray*}
\nonumber q_1 := -a_1 a_2 a_3 a_4 a_5 a_6 \bigg{(}\bigg{(}a_1-\frac{1}{a_1}\bigg{)}\bigg{(}a_4-\frac{1}{a_4}\bigg{)} + \bigg{(}a_2-\frac{1}{a_2}\bigg{)}\bigg{(}a_5-\frac{1}{a_5}\bigg{)}\\
+\bigg{(}a_3-\frac{1}{a_3}\bigg{)}\bigg{(}a_6-\frac{1}{a_6}\bigg{)}\bigg{)},
\end{eqnarray*}
\begin{eqnarray*}
\nonumber
q_2 := a_1 a_2 a_3 a_4 a_5 a_6 (a_1 a_4 + a_2 a_5 + a_3 a_6 + a_1a_2a_6 + a_1a_3a_5 + a_2a_3a_4 + \\a_4a_5a_6
+ a_1a_2a_3a_4a_5a_6).
\end{eqnarray*}

Given a function $f(x,y,\dots,z)$, let $f(x,y,\dots,z)\mid^{z=z_{-}}_{z=z_{+}}$ denote the difference $f(x,y,\dots,z_{-}) - f(x,y,\dots,z_{+})$. Now we define the following 
function $\mathscr{V} = \mathscr{V}(a_1,a_2,a_3,a_4,a_5,a_6,z)$ by means of the equality
\begin{equation*}
\mathscr{V} := \frac{i}{4}\left( \mathscr{U}(a_1,a_2,a_3,a_4,a_5,a_6,z) - z\, \frac{\partial \mathscr{U}}{\partial z}\, \log z \right)\bigg{\vert}^{z=z_{-}}_{z=z_{+}}.
\end{equation*}

\begin{proposition}\label{prop-volume}
The volume of a prism truncated tetrahedron $T$ is given by
\begin{equation*}
\mathrm{Vol}\,T = \Re\left( -\mathscr{V} + a_4 \frac{\partial \mathscr{V}}{\partial a_4}\,\log a_4 \right).
\end{equation*}
\end{proposition}
\begin{proof}
Let us denote
\begin{equation*}
f(T) = \Re\left( -\mathscr{V} + a_4 \frac{\partial \mathscr{V}}{\partial a_4}\,\log a_4 \right),
\end{equation*}
and compute the derivative
\begin{equation*}
\frac{\partial}{\partial \ell}\left( f(T) + \frac{\mu\,\ell}{2} \right) = a_4\,\frac{\partial}{\partial a_4}\left( f(T) + \frac{\mu\,\log a_4}{2} \right) = 
\end{equation*}
\begin{equation*}
= a_4\,\frac{\partial}{\partial a_4}\left( \Re\left( -\mathscr{V} + \left( a_4\,\frac{\partial \mathscr{V}}{\partial a_4} + \frac{\mu}{2} \right)\, \log a_4 \right) \right).
\end{equation*}

The function $\Re \left( a_4\,\frac{\partial \mathscr{V}}{\partial a_4} + \frac{\mu}{2} \right)$ has an a.e. vanishing derivative, c.f. the note in \cite{KM2012} after Theorem 1 saying that $\mu \equiv -2\, \Re (a_4\,\frac{\partial \mathscr{V}}{\partial a_4}) \mod \pi$. Hence,
\begin{equation*}
\frac{\partial}{\partial \ell}\left( f(T) + \frac{\mu\,\ell}{2} \right) = a_4\,\frac{\partial}{\partial a_4}\left( \Re\left( -\mathscr{V} + \left( a_4\,\frac{\partial \mathscr{V}}{\partial a_4} + \frac{\mu}{2} \right)\, \log a_4 \right) \right) \overbrace{=}^{(1)}
\end{equation*}
\begin{equation*}
\overbrace{=}^{(1)} \Re\left( -a_4\,\frac{\partial\mathscr{V}}{\partial a_4} + a_4\,\frac{\partial\mathscr{V}}{\partial a_4} + \frac{\mu}{2} \right) = \frac{\mu}{2}.
\end{equation*}
The equality (1) holds because of the commutativity of the operations $\Re$ and $\frac{\partial}{\partial a_4}$ for the function $-\mathscr{V} + \left( a_4\,\frac{\partial \mathscr{V}}{\partial a_4} + \frac{\mu}{2} \right)\, \log a_4$. The latter holds since $a_4 = e^\ell$ is a real parameter. 

This implies that $\frac{\partial f(T)}{\partial \mu} = -\frac{\ell}{2}$. By analogy to the proof of \cite[Theorem~1]{KM2012}, we can show that $\frac{\partial f(T)}{\partial \theta_k} = -\frac{\ell_k}{2}$, and that if $T$ degenerates into a right Euclidean prism, then $f(T) \rightarrow 0$. Thus, $\mathrm{Vol}\,T = f(T)$ and the proposition follows. 
\end{proof}

Also, we have the following way to determine the dihedral angle $\mu$ along the length $\ell$ edge coming from the intense truncation.

\begin{proposition}\label{prop-angle}
The angle $\mu$ is given by 
\begin{equation*}
\mu \equiv -\Re \left( \frac{i\,a_4}{2}\, \frac{\partial\mathscr{U}(a_1,\dots,a_6,z)}{\partial a_4}\bigg{\vert}^{z=z_{-}}_{z=z_{+}} \right)\, \mod \pi.
\end{equation*}
\end{proposition}
\begin{proof}
We have $\mu \equiv - 2 \Re\left( a_4\,\frac{\partial \mathscr{V}}{\partial a_4} \right)$ mod $\pi$, where $0 < \mu < \pi$ and has an a.e. vanishing derivative.  

Then we compute 
\begin{equation*}
\frac{\partial \mathscr{U}(a_1,\dots,a_6,z_{\pm}(a_1,\dots,a_6))}{\partial a_4} - \frac{\partial}{\partial a_4}\left( z_{\pm}\,\frac{\partial \mathscr{U}(a_1,\dots,a_6,z_{\pm})}{\partial z}\,\log z_{\pm} \right) = 
\end{equation*}
\begin{eqnarray*}
\nonumber
\frac{\partial \mathscr{U}(a_1,\dots,a_6,z_{\pm})}{\partial a_4} + \frac{\partial z_{\pm}}{\partial a_4}\,\frac{\partial \mathscr{U}(a_1,\dots,a_6,z_{\pm})}{\partial z}\\ 
&\hspace*{-0.77in}- \frac{\partial z_{\pm}}{\partial a_4}\,\frac{\partial \mathscr{U}(a_1,\dots,a_6,z_{\pm})}{\partial z} =\\
\frac{\partial \mathscr{U}(a_1,\dots,a_6,z_{\pm})}{\partial a_4}, 
\end{eqnarray*}
since, for some $m\in \mathbb{Z}$,
\begin{equation*}
z_{\pm}\,\frac{\partial \mathscr{U}(a_1,\dots,a_6,z_{\pm})}{\partial z} = 2\pi\, i\, m,
\end{equation*} 
by the definition of $z_{-}$ and $z_{+}$.

Therefore, we obtain
\begin{equation*}
\mu \equiv - 2 \Re \left( a_4\, \frac{\partial \mathscr{V}}{\partial a_4} \right) \hspace*{-0.1in} \mod \pi \equiv - \Re \left( \frac{i\,a_4}{2}\, \frac{\partial\mathscr{U}(a_1,\dots,a_6,z)}{\partial a_4}\bigg{\vert}^{z=z_{-}}_{z=z_{+}} \right) \hspace*{-0.1in} \mod \pi,
\end{equation*}
where $0 < \mu < \pi$. 
\end{proof}

\section{Numerical examples}

Finally, we produce some numerical examples concerning an $n$-gonal ($n\geq 5$) prism $\Pi_n$ with the following distribution of dihedral angles: the angles along the vertical edges are $\frac{2\pi}{5}$, the angles adjacent to the bottom face are $\frac{\pi}{3}$, and those adjacent to the top face are $\frac{\pi}{2}$. Indeed, such a prism $\Pi_n$ exists due to \cite[Theorem~1.1]{HR}. Then we apply Theorem~\ref{thm_prism_volume} for the cases $n=5,6,7$, and perform all necessary numeric computations with Wolfram Mathematica\textsuperscript{\textregistered}. 

In order to avoid excessive branching in numerical computations, we use the modified parameters
\begin{equation*}
q^\prime_i := \frac{q_i}{\Pi^{6}_{k=1}\, a_k}\,\, \mbox{ and }\,\, z_{\pm} := \frac{-q^\prime_1-\sqrt{q^{\prime 2}_1 \pm 4q^\prime_0q^\prime_2}}{2q^\prime_2}. 
\end{equation*}
in the formulae for $\mathscr{U}$ and $\mathscr{V}$ from Section \ref{volume-modified}. 

It follows from the definition of $q^\prime_i$, $i=1,2,3$, above that the quantity $q^{\prime 2}_1-4q^\prime_0q^\prime_2$ is a real number, c.f. \cite[Section 1.1, Lemma]{MY}. This fact prevents computational discrepancies and simplifies any further numerical analysis of the volume formula. 

\begin{table}[ht]
\begin{center}
\begin{tabular}{|c|c|c|}
\hline 
$n$& $(\ell^{\star}, \mu)$& $\mathrm{Vol}\,\Pi_n$ \\
\hline
$5$& $(0.50672, 2\pi/5)$& $2.63200$\\
$6$& $(0.38360, \pi/3)$& $3.43626$\\
$7$& $(0.312595, 2\pi/7)$& $4.19077$\\
\hline
\end{tabular}
\caption{Left: parameters $(\ell^{\star},\mu)$ of $T_n$, right: volume of $\Pi_n$}
\end{center}
\end{table}

Each of the above prisms $\Pi_n$ can be subdivided into $n$ isometric copies of a prism truncated tetrahedron $T_n$. Indeed, $T_n$ is a prism truncated tetrahedron with angles $\theta_1 = \frac{2\pi}{5}$, $\theta_2 = \theta_3 = \frac{\pi}{2}$, $\theta_5 = \theta_6 = \frac{\pi}{3}$, and $\mu = \frac{2\pi}{n}$. By rotating it along the edge with dihedral angle $\mu$, we compose the desired prism $\Pi_n$. 

The graph of $\mathrm{Vol}\,T_n$, with $n=5$, as a function of $\ell$, is shown in Fig.~\ref{fig:Prism-Test} on the left. The graph of $\frac{\partial\Phi}{\partial \ell}(\ell)$ for the same prism truncated tetrahedron $T_n$ is depicted in Fig.~\ref{fig:Prism-Test} on the right. We observe that the function $\frac{\partial\Phi}{\partial \ell}(\ell)$ is indeed monotone and has a single zero $\ell^{\star} \approx 0.50672...$.

\begin{figure}[ht]
\centering
\begin{minipage}{5.5cm}
\includegraphics* [totalheight=3.0cm]{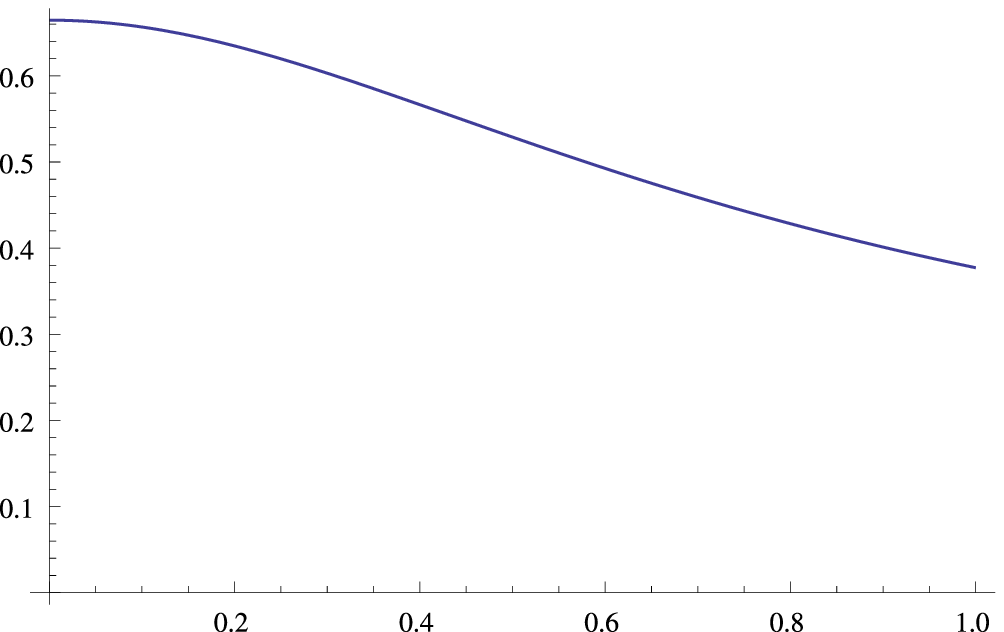}
\end{minipage}
\qquad
\begin{minipage}{5.5cm}
\includegraphics* [totalheight=3.0cm]{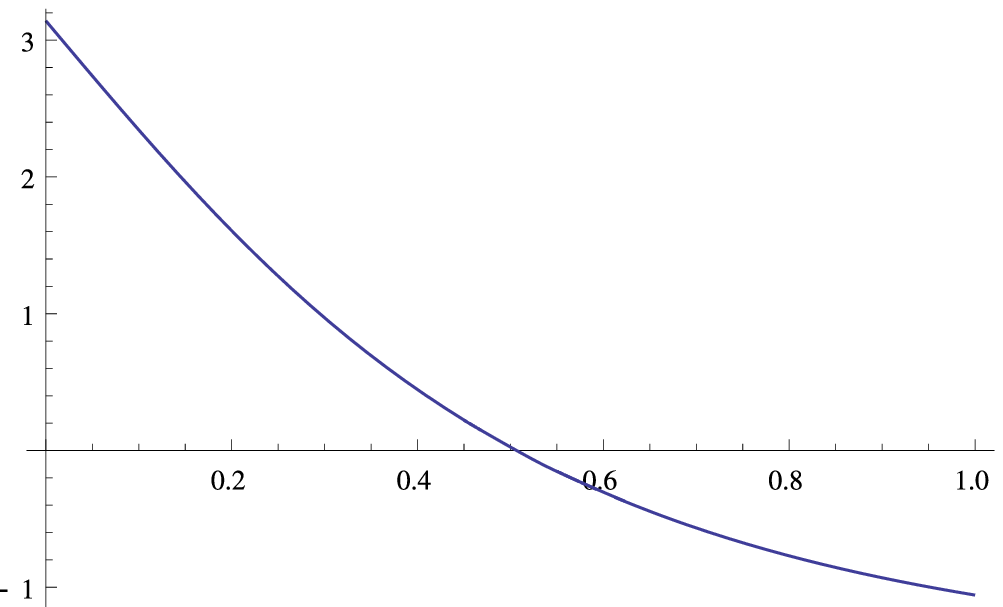}
\end{minipage}
\caption{Left: $\mathrm{Vol}\,T_5$, right: $\frac{\partial\Phi}{\partial \ell}$, both as functions of $\ell$}
\label{fig:Prism-Test}
\end{figure}

The volume of $T_5$ with $\theta_1 = \frac{2\pi}{5}$, $\theta_2 = \theta_3 = \frac{\pi}{2}$, $\theta_5 = \theta_6 = \frac{\pi}{3}$ and $\ell^{\star} \approx 0.50672...$ equals $\sim 0.52639...$ by Proposition~\ref{prop-volume}. Thus, we can see that $\mathrm{Vol}\,\Pi_5 = 5\cdot \mathrm{Vol}\, T_5$ in accordance with Theorem~\ref{thm_prism_volume}, and from Proposition~\ref{prop-angle} $\mu = 1.25664... \approx \frac{2\pi}{5}$.

\newpage

\begin{flushleft}
\textit{
Alexander Kolpakov\\
Department of Mathematics\\
University of Toronto\\
40 St. George Street\\
Toronto ON\\
M5S 2E4 Canada\\
kolpakov.alexander(at)gmail.com}
\end{flushleft}

\medskip

\begin{flushleft}
\textit{
Jun Murakami\\
Department of Mathematics\\
Faculty of Science and Engineering\\
Waseda University\\
3-4-1 Okubo Shinjuku-ku\\ 
169-8555 Tokyo, Japan\\
murakami(at)waseda.jp}
\end{flushleft}


\begin{thebibliography}{}
\bibitem{Buser}\textsc{P.~Buser} {``Geometry and spectra of compact Riemann surfaces''}, New-York, Heidelberg, London: Springer, 2010. 

\bibitem{Cho}\textsc{Y.~Cho} {``Trigonometry in extended hyperbolic space and extended de Sitter space''}, Bull. Korean Math. Soc. \textbf{46}~(6) 1099-1133 (2009); arXiv:0712.1877.

\bibitem{ChoKim}\textsc{Y.~Cho, H.~Kim} {``On the volume formula for hyperbolic tetrahedra''},  Discrete Comput. Geom. \textbf{22}~(3), 347-366~(1999).

\bibitem{DK}\textsc{D.A.~Derevnin, A.C.~Kim} {``The Coxeter prisms in $\mathbb{H}^3$''} in Recent advances in group theory and low-dimensional topology (Heldermann, Lemgo, 2003), pp.~35-49.

\bibitem{DM}\textsc{D.A.~Derevnin, A.D.~Mednykh} {``A formula for the volume of a hyperbolic tetrahedron''}, Russ.~Math.~Surv., \textbf{60}~(2), 346-348~(2005).

\bibitem{F}\textsc{W.~Fenchel} {``Elementary geometry in hyperbolic space''}, Berlin, New-York: Walter de Gruyter, 1989.

\bibitem{Guo}\textsc{R.~Guo} {``Calculus of generalized hyperbolic tetrahedra''}, Geometriae Dedicata, \textbf{153}~(1), 139-149~(2011); arXiv:1007.0453.

\bibitem{HR}\textsc{C.D.~Hodgson, I.~Rivin} {``A characterization of compact convex polyhedra in hyperbolic $3$-space''}, Invent. Math. \textbf{111}, 77-111 (1993). 

\bibitem{LuoGuo}\textsc{R.~Guo, F.~Luo} {``Rigidity of polyhedral surfaces - II''}, Geom. Topol. \textbf{13}, 1265-1312 (2009); arXiv:0711.0766.

\bibitem{Kashaev}\textsc{R.M.~Kashaev} {``The hyperbolic volume of knots from the quantum dilogarithm''}, Lett. Math. Phys. \textbf{39}~(3), 269-275 (1997); arXiv:q-alg/9601025.

\bibitem{K}\textsc{R.~Kellerhals} {``On the volume of hyperbolic polyhedra''},  Math. Ann., \textbf{285}~(4), 541-569 (1989).

\bibitem{KM2012}\textsc{A.~Kolpakov, J.~Murakami} {``Volume of a doubly truncated hyperbolic tetrahedron''}, Aequationes Math., \textbf{85}~(3), 449-463 (2013); arXiv:1203.1061.

\bibitem{KM-err}\textsc{A.~Kolpakov, J.~Murakami} {``Erratum to: Volume of a doubly truncated hyperbolic tetrahedron''}, Aequationes Math., \textbf{88}~(1-2), 199-200 (2014).

\bibitem{KMP}\textsc{A.~Kolpakov, A.~Mednykh, M.~Pashkevich} {``Volume formula for a $\mathbb{Z}_2$-symmetric spherical tetrahedron through its edge lengths''}, Arkiv f\"{o}r Matematik, \textbf{51}~(1), 99-123~(2013); arXiv:1007.3948.

\bibitem{Luo}\textsc{F.~Luo} {``3-dimensional Schl\"{a}fli formula and its generalization''}, Commun. Contemp. Math., \textbf{10}, suppl.~1, 835-842 (2008); arXiv:0802.2580.

\bibitem{Milnor}\textsc{J.~Milnor} {``The Schl\"{a}fli differential equality''} in Collected Papers. I. Geometry (Publish or Perish, Houston, TX, 1994), pp.~281-295.

\bibitem{M2012}\textsc{J.~Murakami} {``The volume formulas for a spherical tetrahedron''}, Proc. Amer. Math. Soc. \textbf{140}~9, 3289-3295 (2012); arXiv:1011.2584.

\bibitem{HMJM}\textsc{H.~Murakami, J.~Murakami} {``The colored Jones polynomials and the simplicial volume of a knot''}, Acta Math. \textbf{186}~(1), 85-104 (2001); arXiv:math/9905075.

\bibitem{MU}\textsc{J.~Murakami, A.~Ushijima} {``A volume formula for hyperbolic tetrahedra in terms of edge lengths''},  J.~Geom. \textbf{83}~(1-2), 153-163~(2005); arXiv:math/0402087.

\bibitem{MY}\textsc{J.~Murakami, M.~Yano} {``On the volume of hyperbolic and spherical tetrahedron''}, Comm. Anal. Geom. \textbf{13}~(2), 379-400~(2005).

\bibitem{Ratcliffe}\textsc{J.G.~Ratcliffe}, {``Foundations of hyperbolic manifolds''}, New York: Springer-Verlag, 1994. (Graduate Texts in Math.; 149).

\bibitem{U}\textsc{A.~Ushijima} {``A volume formula for generalised hyperbolic tetrahedra''} in Mathematics and Its Applications \textbf{581} (Springer, Berlin, 2006), pp.~249-265.
\end{thebibliography}
\end{document}